\documentclass[12pt]{amsart}
\usepackage{amsmath}
\usepackage{amssymb,latexsym}
\usepackage{amsthm}
\usepackage{amsbsy}
\usepackage{amscd}
\usepackage{amsrefs}
\usepackage{enumerate}

\makeatletter
\@namedef{subjclassname@2010}{%
  \textup{2010} Mathematics Subject Classification}
\makeatother
\newtheorem{thm}{Theorem}[section]
\newtheorem{cor}[thm]{Corollary}
\newtheorem{lem}[thm]{Lemma}
\newtheorem{prop}[thm]{Proposition}

\theoremstyle{definition}

\theoremstyle{remark}
\newtheorem{rem}[thm]{Remark}
\newtheorem{exa}[thm]{Example}
\theoremstyle{conjecture}

\numberwithin{equation}{section}

\begin{document}

\title[Differentiation and integration operators]{Differentiation and integration operators on weighted Banach spaces of holomorphic functions}

\author[A.V. Abanin]{Alexander V. Abanin}
\address{Southern Federal University\\ Rostov-on-Don 344090\\
and Southern Institute of Mathematics\\ Vladikavkaz 362027\\ Russian Federation}
\email{abanin@math.rsu.ru}

\author[Pham Trong Tien]{Pham Trong Tien}
\address{Hanoi University of Science, VNU \\ 334 Nguyen Trai, Thanh Xuan, Ha Noi\\ Viet Nam}
\email{phamtien@mail.ru}



\begin{abstract}
 We obtain a new natural description of the class of radial weights for which some previous results concerning the boundedness of differentiation and integration operators on corresponding spaces are valid. To do this, we develop a new elementary approach which is essentially different from the previous one and can be applied for weights and domains of general types. We also establish a new characterization of some popular classes of radial weights.
\end{abstract}




\maketitle

\section{Introduction} 
Let $G$ be a domain in the complex plane $\mathbb{C}$ and $H(G)$ the space of all holomorphic functions on $G$.
For a positive continuous function (a \textit{weight}) $v$ on $G$ define the Banach space
$$
H_v(G): = \{f \in H(G): \|f\|_v: = \sup_{z \in G} \dfrac{|f(z)|}{v(z)} < \infty \},
$$
endowed with the norm $\|\cdot\|_v$.
Spaces of such a type and classical operators between them play an important role in analysis and its applications (see e.g. Bonet's review \cite{B-03}).
From this reason they have been studied by many authors (see e.g. \cite{AT-12}, \cite{BBF-13}, \cite{BBF-15}, \cite{BBT-98}, \cite{BS-93} \cite{B-03}, \cite{B-09}, \cite{BB-13}, \cite{BDL-99}, \cite{BDLT-98}, \cite{BW-03}, \cite{DoLi-02}, \cite{HL-08}, \cite{L-92}, \cite{L-95}, \cite{L-96}, \cite{L-06}).
In particular, Harutyunyan and Lusky \cite{HL-08} have investigated the boundedness of the differentiation and integration operators in the cases when $G$ is either the whole complex plane $\mathbb{C}$ or the open unit disk $\mathbb{D} = \{z \in \mathbb{C}: |z| < 1\}$ and $v$ a radial weight.
Later on, their results were used in \cite{BBF-13}, \cite{BBF-15}, \cite{B-09}, \cite{BB-13} as a starting point for the study of some important properties of these two operators.
Recall that a \textit{ radial weight } on $G=\mathbb{C}$ or $G=\mathbb{D}$ is a positive function $v$ on $G$ with $v(z)=v(|z|)$, $z\in G$, where $v(r)$ is continuous and increasing on $[0,a)$ and
\[
  \log r = o(\log v(r)) \text{ as } r \to \infty \text{ if } G =\mathbb{C} \ \ \text{ and } \ \ \lim_{r \to 1^-}v(r) = \infty \text{ if } G = \mathbb{D}.
\]
Here and below $a=+\infty$ for $G=\mathbb C$ and $a=1$ for $G=\mathbb D$.
In some proofs in \cite{HL-08} it was used a rather complicated technique developed by Lusky in \cite{L-06} to give an isomorphic classification of spaces $H_v(\mathbb C)$ and $H_v(\mathbb D)$ defined by radial weights.
Moreover, all the main results in \cite{HL-08} were obtained under the following additional assumption:
\begin{center}
(HL)
\textit{ any $r$ in $[0,a)$ is a global maximum point for some function $\gamma_n(r): = r^n / v(r), n\in(0,\infty)$},
\end{center}
which looks hard to check. It should be noted that the results in \cite{HL-08} were used in the papers \cite{BBF-13}, \cite{BBF-15}, \cite{B-09}, \cite{BB-13} without mentioning this assumption (HL).
In view of these reasons, we develop an essentially new elementary method to study the boundedness of the classical operators in Banach weighted spaces of holomorphic functions.
It suits to all domains and weights and, in the radial case, allows us to remove some additional restrictions used in \cite{HL-08} (see Remark~\ref{rem-l} below).

Our idea is based on the following observation: as is well known (see e.g. \cite{BBT-98} and \cite{AT-12}), to characterize properties of spaces $H_v(G)$ and operators between them in terms of weights, one cannot, in general, use the initial weights defining the spaces but should use so-called associated or essential weights.
Recall (see Bierstedt-Bonet-Taskinen \cite{BBT-98}) that, given a weight $v$ on $G$, its \textit{associated weight} is defined by
$$
\widetilde{v}(z): = \sup\{|f(z)|: f \in H_v(G), \|f\|_v \leq 1 \}.
$$
Note that, for a nontrivial $H_v(G)$,   $0<\widetilde v\le v$ and $\log \widetilde{v}$ is a continuous subharmonic function on $G$ and $H_{\widetilde{v}}(G) = H_v(G)$ isometrically.
We say that a weight $v$ on $G$ is \textit{essential} if $v$ is equivalent to $\widetilde v$, that is, there is a constant $C\ge 1$ such that  $v(z) \leq C \widetilde v(z) \text{ for all } z \in G$.
It is easy to see that for a radial weight $v$ its associated weight $\widetilde{v}$ is also radial and $\widetilde{v}(r)$ is increasing and \textit{log-convex} (i.e. the function $\log\widetilde{v}(e^x)$ is convex), which is equivalent to the subharmonicity of $\log\widetilde v(z)$.
Thus, if one hope to obtain a complete description of some topological properties of spaces $H_v(G)$ or linear operators between them in terms of weights in  the radial case, he should use log-convex weights.

In Section~2 we establish some necessary and sufficient conditions for the boundedness of the differentiation operator $D: f\mapsto f'$ in both general and radial cases.
Then, under the natural assumption that $v$ is a log-convex weight, we give new proofs of criteria for the boundedness of the operator $D: H_v(G) \to H_w(G)$, where $w(r)= v(r)/(1-r)$ on $G = \mathbb{D}$ and $w = v$ on $G = \mathbb{C}$. Note that these criteria contain the following new characterization of bounded differentiation operators:\textit{ The differentiation operator as above is continuous if and only if $v(r)=O(v(r^2))$ as $r\to1^-$ for $\mathbb D$ and $\log v(r)=O(r)$ as $r\to\infty$ for $\mathbb C$}. This is also a new description of weights satisfying some popular conditions (see Doma\'nski--Lindst\"om \cite{DoLi-02} and references therein). In addition, we construct two examples showing that the log-convexity of weights is essential for these results.

In Section~3, similarly to Section~2, we study  the integration operator $I: f\mapsto \int_0^z f(t)\,dt$. In particular, we obtain a complete description of those log-convex weights $v$ for which the integration operator $I:H_v(\mathbb C)\to H_v(\mathbb C)$ is bounded. In the case of the unit disc the situation is more complicated. We establish criteria for only some classes of weights introduced before in Horowitz \cite{Hor-95} and Bierstedt--Bonet--Taskinen \cite{BBT-98} or satisfying some additional assumptions.

\section{Boundedness of the differentiation operator}
In this section we study the differentiation operator $D: H_v(G) \to H_w(G)$ via a new method which is absolutely different from the one of Harutyunyan and Lusky \cite{HL-08}. In particular, we prove that the main results in \cite[Theorems~3.2 and 4.1]{HL-08} are true for log-convex weights. Moreover, we give some examples which show that the use of the log-weights is essential for these results (see Examples \ref{ex-1} and \ref{ex-2}).

We start with some simple necessary and sufficient conditions for the boundedness  of the differentiation operator on weighted spaces of a general type.

\begin{prop}\label{prop-ne}
Let $v$ and $w$ be two weights on $G$. If the differentiation operator $D:H_v(G) \to H_w(G)$ is bounded, then there is a constant $C>0$ such that
\begin{equation}\label{eq-1}
\limsup_{\Delta z \to 0} \left| \dfrac{\widetilde{v}(z+\Delta z)-\widetilde{v}(z)}{\Delta z} \right| \leq C w(z) \text{ for all } z \in G.
\end{equation}
\end{prop}
\begin{proof}
If $D:H_v(G) \to H_w(G)$ is continuous, then there is a constant $C>0$ such that
$ \|f'\|_w \leq C$  for all $f$ in $B_v(G)$,  the unit ball of $H_v(G)$.

Fix arbitrary $z \in G$ and $\Delta z \in \mathbb{C}$ so that the closed interval $[z, z+ \Delta z]$ is contained in $G$.
Without loss of generality we may assume that $\widetilde{v}(z+\Delta z) \geq \widetilde{v}(z)$.
By Bierstedt, Bonet and Taskinen \cite[1.2 (iv)]{BBT-98}, there exists a function $f$ in $B_v(G)$ with
$f(z+\Delta z) = \widetilde{v}(z+\Delta z)$.
Then we have
\begin{align*}
|\widetilde{v}(z+\Delta z) - \widetilde{v}(z)| & = \widetilde{v}(z+\Delta z) - \widetilde{v}(z) \leq f(z+\Delta z) - |f(z)| \\
& \leq |f(z+\Delta z) - f(z)| \leq |f'(\xi)||\Delta z|
\end{align*}
for some point $\xi \in [z,z+\Delta z]$.
Hence, $|\widetilde{v}(z+\Delta z) - \widetilde{v}(z)| \leq C w(\xi) |\Delta z|$, which implies \eqref{eq-1}.
\end{proof}

 Consider now the special cases when $G = \mathbb{D}$ or $G = \mathbb{C}$ and weights are radial.
  As it was noted above, for a radial weight $v$ the function $\widetilde{v}(r)$ is increasing and  log-convex.
 Consequently, $\widetilde{v}$ is differentiable on $(0,\infty)$ except not more than countably many points and its right derivative exists everywhere on $(0,\infty)$ and increases.
 In what follows we will use the notation $\widetilde v'$ for the right derivative of $\widetilde v$.
 From Proposition~\ref{prop-ne} it follows immediately:

\begin{cor}\label{cor-ne}
Let $v$ and $w$ be two radial weights on $G=\mathbb{D}$ ( resp., $G=\mathbb{C}$). If the differentiation operator $D:H_v(G) \to H_w(G)$ is continuous, then
\begin{equation}\label{eq-2}
\limsup_{r\to 1^-} \dfrac{\widetilde{v}'(r)}{w(r)} <\infty\ \  \left(\text{resp., } \limsup_{r\to\infty} \dfrac{\widetilde{v}'(r)}{w(r)} <\infty\right).
\end{equation}
\end{cor}

In \cite[Theorems~2.1(b)]{HL-08} a similar result was established for weights satisfying additional assumption (HL) but without using in \eqref{eq-2} the associated weight $\widetilde{v}$.

A continuous function $\rho: G \rightarrow (0,1]$ is called a \textit{distance function} on $G$ if $\rho(z) < dist(z,\partial{G})$ for each $z$ in $G$.
For a weight $v$ and a distance function $\rho$ on $G$ we define the following new weight
$$
v_{\rho}(z): = \dfrac{1}{\rho(z)}\max_{|\zeta| \leq \rho(z)} v(z+\zeta),\ z \in G.
$$
Obviously, $v(z) \leq v_{\rho}(z)$ for all $z \in G$.
\begin{prop}
For any weight $v$ and distance function $\rho$ on $G$ the differentiation operator $D: H_v(G) \to H_{v_{\rho}}(G)$ is continuous.
\end{prop}
\begin{proof}
 By the classical Cauchy  formula we have that
$$
f'(z) = \dfrac{1}{2\pi i}\int_{|\zeta| = \rho(z)} \dfrac{f(z+\zeta)}{\zeta^2}\,d\zeta,\ z \in G,
$$
for any holomorphic function $f$ on $G$.
Then, for all $z \in G$ and $f \in H_v(G)$,
$$
|f'(z)| \leq \dfrac{1}{\rho(z)} \|f\|_v \max_{|\zeta| \leq \rho(z)} v(z+\zeta) = \|f\|_v v_{\rho}(z).
$$
Hence, $\|f'\|_{v_{\rho}} \leq \|f\|_v$ for all $f \in H_v(G)$, which completes the proof.
\end{proof}
\begin{cor}\label{cor-su}
Let $v$ and $w$ be two weights and $\rho$ a distance function on $G$. If there is a constant $C>0$ such that $v_{\rho}(z) \leq C w(z)$ for all $z \in G$, then the operator $D: H_v(G) \to H_w(G)$ is continuous.
\end{cor}

Applying Corollary \ref{cor-su} to radial weights on $\mathbb{D}$ (resp. $\mathbb{C}$) with the distance function $\rho(z) =(1-|z|)/2$ ( resp., $\rho(z)=1$), we immediately obtain the following result.
\begin{cor}\label{cor-su-r}
The following assertions are valid.
\begin{enumerate}
\item[1.] Let $v$ and $w$ be two radial weights on $\mathbb{D}$ such that for some constant $C>0$
$$
\dfrac{1}{1-r}v\left(\dfrac{1+r}{2}\right) \leq C w(r) \text{ for all } r \in [0,1).
$$
Then the operator $D: H_v(\mathbb{D}) \to H_w(\mathbb{D})$ is continuous.
\item[2.] Let $v$ and $w$ be two radial weights on $\mathbb{C}$ such that for some constant $C>0$
$$
v(1+r) \leq C w(r) \text{ for all } r \in [0,\infty).
$$
Then the operator $D: H_v(\mathbb{C}) \to H_w(\mathbb{C})$ is continuous.
\end{enumerate}
\end{cor}

Now we consider special cases when $w(r) = v(r)/(1-r)$ on $\mathbb{D}$ and $w = v$ on $\mathbb{C}$ and show that the previous results in these cases remain valid for log-convex weights.
To do this, we need the following auxiliary result concerning growth conditions on a radial weight.
\begin{lem}\label{lem-d}
For a radial weight $v$ on $\mathbb{D}$ consider the following conditions:
\begin{enumerate}
\item[(i)] $\displaystyle \limsup_{ r \to 1^-} \dfrac{(1-r)v'(r)}{v(r)} < \infty$.
\item[(ii)] $(1-r)^\alpha v(r)$ is decreasing on $[r_0,1)$ for some $\alpha>0$ and $r_0\in[0,1)$.
\item[(iii)] $(1-r)^\alpha v(r)$ is almost decreasing on $[0,1)$ for some $\alpha>0$, i.e. there exists a constant $C>0$ such that for any $r_1<r_2$ it follows that $(1-r_2)^\alpha v(r_2)\le C(1-r_1)^\alpha v(r_1)$.
\item[(iv)] $1/v$ satisfies condition $(*)$ from \cite[p.~310]{L-95} and \cite[Definition~2.1]{L-96}, i.e.
\[
 \sup_{n} \dfrac{v(1-2^{-n-1})}{v(1-2^{-n})} < \infty.
\]
\item[(v)]
There exists $\delta_0\in(0,1)$ such that
\[
v\left(\dfrac{r+\delta_0}{1+\delta_0r}\right)=O(v(r)) \ \text{ as } r\to1^{-}.
\]
\item[(vi)]
\[
v(r)=O(v(r^2)) \ \text{ as } r\to1^{-}.
\]
\end{enumerate}
Then (i)$\Longleftrightarrow$(ii)$\Longrightarrow$(iii)$\Longleftrightarrow$
(iv)$\Longleftrightarrow$(v)$\Longleftrightarrow$(vi).
For a log-convex weight $v$ conditions (i)--(vi) are all equivalent.
\end{lem}
\begin{proof}
(i)$\Longleftrightarrow$(ii): Similarly to \cite[Theorem~3.2]{HL-08}.

(ii)$\Longrightarrow$(iii): Obvious.

(iii)$\Longleftrightarrow$(iv): See \cite[Lemma~1 (a)]{DoLi-02}.

(iv)$\Longleftrightarrow$(v): Follows from the proof of Lemma~1 in \cite{DoLi-02} (see (a): equivalence (iii)$\Longleftrightarrow$(v)).

(v)$\Longleftrightarrow$(vi):
Given $\delta_0\in(0,1)$, choose $\gamma\in(0,1)$ so that $(1-\gamma)/(1+\gamma)<\delta_0$.
Since
$$
r^\gamma\frac{1-r^{1-\gamma}}{1-r^{1+\gamma}}\to\frac{1-\gamma}{1+\gamma}\ \text{ as } r\to1^-,
$$
there exists $r_0\in(0,1)$ such that, for all $r\in(r_0,1)$,
$$
\delta_0\ge
r^\gamma\frac{1-r^{1-\gamma}}{1-r^{1+\gamma}} \ \text{ or }
r^\gamma\le \frac{r+\delta_0}{1+\delta_0r}.
$$
Therefore, if (v) holds, then $v(r^\gamma)=O(v(r))$ as $r\to1^-$, which implies (vi).

Conversely, putting $\gamma=1/2$ and $\delta_0=1/4<(1-\gamma)/(1+\gamma)$, we see that, for all $r$ close to 1 on the left,
$$
\sqrt{r}\ge\dfrac{r+\frac14}{1+\frac14r}.
$$
From this it follows easily that (vi) implies (v) with $\delta_0=1/4$.

Let now $v$ is log-convex.

(iv)$\Longrightarrow$ (i).
Given $r\in[1/2,1)$, choose $n\in\mathbb N$ so that
$1-2^{-n} \leq r < 1-2^{-n-1}$.
Using that $v$ is log-convex, we have
\begin{align*}
r\,(\log v(r))'&\le
(1-2^{-n-1})(\log v)'(1-2^{-n-1}) \\
& \leq
\dfrac{\log(v(1-2^{-n-2})/v(1-2^{-n-1}))}{\log((1-2^{-n-2})/(1-2^{-n-1}))} \\
&\le \dfrac{\log A}{\log((1-2^{-n-2})/(1-2^{-n-1}))}=
\dfrac{\log A}{\log\left(1+\frac{2^{-n-2}}{1-2^{-n-1}}\right)},
\end{align*}
where $A:=\sup\limits_{n} \dfrac{v(1-2^{-n-1})}{v(1-2^{-n})}$.

Since
$$
\dfrac{2^{-n-2}}{1-2^{-n-1}}\ge\dfrac{1-r}{2(1+r)},
$$
it then follows that
$$
(1-r)\dfrac{v'(r)}{v(r)}\le
\dfrac{(1-r)\log A}{r\log\left(1+\dfrac{1-r}{2(1+r)}\right)}\rightarrow 4 \log A
\ \text{ as } r\to1^-.
$$
Thus, (i) holds.
Consequently, for a log-convex weight $v$, conditions (i)--(vi) are all equivalent.
\end{proof}
\begin{rem}\label{r-mgw}
For a log-convex increasing function $v:[0,1)\to(0,\infty)$ condition (iv) of Lemma~\ref{lem-d} means that $\log v$ is of \textit{moderate growth} in the sense of Horowitz \cite[Definition~2]{Hor-95}.
\end{rem}
\begin{thm}\label{thm-d-d}
For a log-convex weight $v$ on $\mathbb{D}$ and $w(r):=v(r)/(1-r)$ assertions (i)--(vi) of Lemma~\ref{lem-d} are equivalent to:
\begin{enumerate}
\item[(vii)] The differentiation operator $D: H_v(\mathbb D)\to H_w(\mathbb D)$ is continuous.
\end{enumerate}
\end{thm}
\begin{proof}
By Lemma~\ref{lem-d}, it is enough to check that condition (vii) implies one of conditions (i)--(vi) and vise versa.

(vii)$\Longrightarrow$ (iv).
 By Corollary \ref{cor-ne}, condition (vii) implies that
\begin{equation}\label{eq-wv}
\limsup_{r \to 1^-} \dfrac{(1-r)\widetilde{v}'(r)}{v(r)} <\infty.
\end{equation}
Next, by \cite[Lemma~5]{BDL-99} (see also \cite[Subsection~4.3]{AD-15}), weights $v$ and $\widetilde v$ are equivalent ($v\sim\widetilde v$), that is, there exists a constant
$C\ge1$ such that $\widetilde{v}(r) \leq v(r) \leq C \widetilde{v}(r)$
for all $r \in [0,1)$.
Consequently, \eqref{eq-wv} is equivalent to
\begin{equation*}
\limsup_{r \to 1^-} \dfrac{(1-r)\widetilde{v}'(r)}{\widetilde{v}(r)} <\infty.
\end{equation*}
Then, by Lemma \ref{lem-d},  $1/\widetilde{v}$ satisfies condition $(*)$:
$$
 \sup_{n} \dfrac{\widetilde{v}(1-2^{-n-1})}{\widetilde{v}(1-2^{-n})} < \infty,
$$
which, in view of $v\sim\widetilde v$, is equivalent to (iv).

(i)$\Longrightarrow$(vii):
From (i) it follows that there is a constant $C>0$ such that
$$
(\log v(r))'\le\dfrac{C}{1-r}\ \ \text{ for all } r\in[0,1).
$$
Using that $v$ is log-convex,  we then have, for all $r \in [1/2,1)$,
\begin{align*}
&\log v\left(\frac{1+r}{2}\right)-\log v(r)\le (\log v)'\left(\frac{1+r}{2}\right)\dfrac{1+r}{2}\log\frac{1+r}{2r} \\
& \leq C \dfrac{1+r}{1-r} \log \left( 1 + \dfrac{1-r}{2r} \right) \leq C\dfrac{1+r}{2r}\le\frac32 C.
\end{align*}
Taking $M:=\max\{e^{3C/2}, v(3/4)/v(0)\}$, we obtain
$$
 v\left( \dfrac{1+r}{2}\right) \leq M v(r)\ \text{ for all } r \in [0,1),
$$
and, consequently,
$$
\dfrac{1}{1-r} v\left( \dfrac{1+r}{2}\right) \leq M w(r)\  \text{ for all } r \in [0,1).
$$
Applying Corollary \ref{cor-su-r}, we obtain (vii).
\end{proof}
Note that Theorem \ref{thm-d-d} might fail for weights of a general type that are not log-convex. To see this, it is enough to consider the following example.

\begin{exa}\label{ex-1}
Let $(a_n)$ and $(b_n)$ be two sequences of positive numbers with the following properties:
\begin{align*}
&(1)\quad \sum_{n=1}^{\infty}(a_n+b_n) =: M<\infty;
&(2)\quad a_n < b_n,\ \forall n \geq 2; \\
&(3)\quad (a_n+b_n)\downarrow 0 \text{ as } n \rightarrow \infty;
&(4)\quad \lim_{n\to \infty} \dfrac{a_n}{b_n}  = 0;\\
&(5)\quad \lim_{n\to \infty} \dfrac{\sum_{k > n}(a_k+b_k)}{a_n+b_n} = L \in (0,+\infty).
\end{align*}
For instance, we can take $a_n = 3^{-n},\ b_n = 2^{-n} - 3^{-n}$ ($n \in \mathbb{N}$).

 Define a function $\varphi: (-\infty,0) \to [0,+\infty)$ in the following way.
Denote
$$
S_n:= - \sum_{k>n}(a_k+b_k), \ \ n \geq 0.
$$
By property (1),  $S_n \uparrow 0$ as $n \to \infty$.
Put
\begin{enumerate}
\item[*] $\varphi(t): = 1$ for $t \in (-\infty, S_0 + a_1]$ and $\varphi(t): = \dfrac{1}{b_1}(t- S_0 - a_1) + 1$ for $t \in (S_0 + a_1,S_1]$;
\item[*] for $n \geq 2$,
$$
\varphi(t) := \dfrac{1}{a_n}(t - S_{n-1})+2n-2, \ \ t\in (S_{n-1},S_{n-1}+a_n],
$$
and
$$
\varphi(t): = \dfrac{1}{b_n} (t - S_{n-1}-a_n) + 2n-1, \ \ t\in (S_{n-1}+a_n, S_n].
$$
\end{enumerate}
Evidently, $\varphi$ is increasing and continuous on $(-\infty,0)$. Next, for every $n \in \mathbb{N}$,
$$
\varphi'(t)= 1/a_n \text{ on } (S_{n-1},S_{n-1}+a_n)
\text{ and }
\varphi'(t)= 1/b_n \text{ on } (S_{n-1}+a_n, S_n).
$$
From this and property (2) it then follows that the function $\varphi$ is not convex on $(-\infty,0)$.
Denote by $\overline{\varphi}$ the largest convex minorant of the function $\varphi$.
Using property (3), it is easy to see that
$$
\overline{\varphi}(t) = \dfrac{2}{a_n+b_n}(t - S_{n-1})+2n-2 \ \text{ for } t\in[S_{n-1},S_n], n\in\mathbb N.
$$

Obviously, the following two functions
$$
v(r):= \exp(\varphi(\log r)) \text{ and }
\overline{v}(r):= \exp(\overline{\varphi}(\log r)),\  r \in [0,1),
$$
are radial weights on $\mathbb D$.

A simple calculation gives that, for every $n \in \mathbb{N}$,
$$
\overline{V}(r):=\dfrac{(1-r)\overline{v}'(r)}{\overline{v}(r)} = \dfrac{2(r^{-1}-1)}{a_n+b_n} \ \text{ for all } r\in (e^{S_{n-1}},e^{S_n}).
$$
Hence,
$$
\max_{(e^{S_{n-1}},e^{S_n})}\overline{V}(r) = \frac{2(e^{-S_{n-1}}-1)}{a_n+b_n}\ \text{ for every } n\in\mathbb N.
$$
Using the definition of $(S_n)$ and property (5), we get
$$
\lim_{n \to \infty}\max_{(e^{S_{n-1}},e^{S_n})}\overline{V}(r) = 2(L+1).
$$
Therefore
$$
\limsup_{r\to 1^-}\dfrac{(1-r)\overline{v}'(r)}{\overline{v}(r)} < \infty.
$$
Taking now  $\overline{w}(r): = \overline{v}(r)/(1-r)$ and applying Theorem \ref{thm-d-d} to the log-convex weight $\overline{v}$, we conclude that the operator $D: H_{\overline{v}}(\mathbb{D}) \to H_{\overline{w}}(\mathbb{D})$ is continuous.
Since $\widetilde{v}(r) \leq \overline{v}(r) \leq v(r)$ on $[0,1)$ and $H_{\widetilde{v}}(\mathbb{D}) = H_{v}(\mathbb{D})$ isometrically (see, e.g., \cite[1.12]{BBT-98}), $H_{\overline{w}}(\mathbb{D})$ is continuously embedded in $H_w(\mathbb{D})$ and $H_{v}(\mathbb{D}) = H_{\overline{v}}(\mathbb{D})$ isometrically.
Consequently, $D: H_{v}(\mathbb{D}) \rightarrow H_w(\mathbb{D})$ is continuous.

On the other hand, for every $n \in \mathbb{N}$,
$$
V(r):=\dfrac{(1-r)v'(r)}{v(r)} = \dfrac{r^{-1}-1}{a_n} \ \text{ for all } r\in (e^{S_{n-1}},e^{S_{n-1}+a_n}).
$$
In particular,
$$
V(e^{S_{n-1}+a_n/2})=
\frac{e^{-S_{n-1}-a_n/2}-1}{a_n}\geq
\frac{e^{b_n}-1}{a_n}\geq
\frac{b_n}{a_n}.
$$
From this and property (4) it then follows that
$$
\limsup_{r\to 1^-}\frac{(1-r)v'(r)}{v(r)}=\infty.
$$
\end{exa}

Consider now the case of entire functions.

\begin{thm}\label{thm-d-c}
Let $v$ be a log-convex weight on $\mathbb{C}$.
Then the following conditions are equivalent:
\begin{enumerate}
\item[(i)] The operator $D: H_v(\mathbb{C}) \to H_v(\mathbb{C})$ is continuous.
\item[(ii)] $\displaystyle \limsup_{ r \to \infty} \dfrac{v'(r)}{v(r)} < \infty$.
\item[(iii)] $\log v(r) = O(r)$ as $r \to \infty$.
\end{enumerate}
\end{thm}
\begin{proof}
(i)$\Longrightarrow$(iii): By Corollary \ref{cor-ne}, (i) implies that
$$
\limsup_{r \to \infty} \dfrac{\widetilde{v}'(r)}{\widetilde{v}(r)} < \infty.
$$
Applying the L'Hospital rule to $f(r):=\log\widetilde{v}(r)$ and $g(r):=r$, we have that
$$
\limsup_{r \to \infty} \dfrac{\log\widetilde{v}(r)}{r} < \infty.
$$
Therefore, $\log \widetilde{v}(r) \leq C (r+1)$ for some $C>0$ and all $r \geq 0$.
As is known (see \cite[Page 157]{BBT-98}), $v(r) \leq r \widetilde{v}(r)$ for all $r \geq 1$ .
Consequently, $\log v(r) \leq C (r+1) + \log r$ for all $r \geq 1$, which  implies  $(iii)$.

(iii)$\Longrightarrow$(ii): Because $v$ is increasing and log-convex, we have that, for every $r>0$,
$$
\dfrac{v'(r)}{v(r)} \leq \dfrac{\log v(er) - \log v(r)}{r} \leq e \dfrac{\log v(er) - \log v(0)}{er}.
$$
From this it follows that
$$
\limsup_{r\to\infty}\dfrac{v'(r)}{v(r)}\le e\limsup_{r\to\infty}\dfrac{\log v(r)}{r}.
$$
Thus, (iii)$\Longrightarrow$(ii).

(ii) $\Longrightarrow$ (i):
(ii) means that there is a constant $C>0$ such that $(\log v(r))'\le C$ for all $r>0$.
Using that $v$ is log-convex, we then have
$$
\log\dfrac{v(r+1)}{v(r)}\le C(r+1)\log\dfrac{r+1}{r}\le 2C \text{ for all } r\ge1.
$$
By Corollary~\ref{cor-su-r}, this implies (i).
\end{proof}
As above, we give an example showing that the use of log-convex weights in our Theorem \ref{thm-d-c} is essential. Since the explanation is similar to the one in Example \ref{ex-1}, we omit it.

\begin{exa}\label{ex-2}
Let $(a_n)$ and $(b_n)$ be two sequences of positive numbers with the following properties:
$$
(1)\ \sum_{n=1}^{\infty}(a_n+b_n) = +\infty; \qquad \quad
(2)\ a_n < b_n,\ \forall n \geq 2; \qquad \quad
$$
$$
(3)\ (a_n+b_n)\downarrow 0 \text{ as } n \rightarrow \infty; \qquad \quad
(4)\ \lim_{n\to \infty} a_n \exp\sum_{k=1}^{n-1}(a_k+b_k) = 0;
$$
$$
(5)\ \lim_{n\to \infty} (a_n+b_n) \exp\sum_{k=1}^{n-1}(a_k+b_k) = L \in (0,+\infty).
$$
For instance, we can take $a_n = 3^{-n}, b_n = \log(1+\frac{1}{n}) - 3^{-n}$ \ ($n \in \mathbb{N}$).

Denote
$$
S_n: = \sum_{k=1}^{n}(a_k+b_k), \ n \in \mathbb{N},
$$
and put
\begin{enumerate}
\item[*] $\varphi(t): = 1$ for $t \in (-\infty, a_1]$ and $\varphi(t): = \dfrac{1}{b_1}(t- a_1) + 1$ for $t \in (a_1,S_1]$;
\item[*] for $n \geq 2$,
$$
\varphi(t): = \dfrac{1}{a_n}(t - S_{n-1})+2n-2, \ \ t\in (S_{n-1},S_{n-1}+a_n],
$$
and
$$
\varphi(t): = \dfrac{1}{b_n} (t - S_{n-1}-a_n) + 2n-1, \ \ t\in (S_{n-1}+a_n, S_n].
$$
\end{enumerate}
 The function $\varphi$ is not convex on $\mathbb{R}$ and its largest convex minorant $\overline{\varphi}$ is defined by
$$
\overline{\varphi}(t): = \dfrac{2}{a_n+b_n} (t - S_{n-1}) + 2n-2,\ t \in [S_{n-1},S_n], n\in\mathbb N.
$$
Since $\overline{\varphi}'(t) = 2/(a_n+b_n)$ on  $(S_{n-1},S_n)$ for every $n \in \mathbb{N}$, $\overline{\varphi}'(t) \to +\infty$ as $t \to +\infty$.
This implies that $t = o(\overline{\varphi}(t))$ as $t \to +\infty$. Hence, $t = o(\varphi(t))$ as $t \to +\infty$, too.

From the above it follows that the functions
$v(r): =\exp(\varphi(\log r))$ and $\overline{v}(r): = \exp(\overline{\varphi}(\log r))$
are radial weights on $\mathbb C$.
Arguing as in Example \ref{ex-1}, we show that by Theorem \ref{thm-d-c} the differentiation operator $D: H_{\overline{v}}(\mathbb{C}) \to H_{\overline{v}}(\mathbb{C})$ or, equivalently, $D: H_{v}(\mathbb{C}) \to H_{v}(\mathbb{C})$ is continuous, but
$$
\limsup_{r \to \infty} \dfrac{v'(r)}{v(r)} = \infty.
$$
\end{exa}

\begin{rem}
Note that by making slight changes of $\varphi$ and $\overline{\varphi}$ in Examples \ref{ex-1} and \ref{ex-2} in the neighborhood of each angular point we may assume that $\varphi$ and $\overline{\varphi}$ are everywhere differentiable, and hence, the corresponding weights $v$ and $\overline{v}$ are weights in the sense of \cite{HL-08}. Then Examples \ref{ex-1} and \ref{ex-2} also show that the results in \cite[Theorems~3.2 and 4.1]{HL-08} cannot be stated for radial weights of a general type.
\end{rem}
\begin{rem}
In this section we actually replaced assumption (HL) by the natural one that $v$ is a log-convex weight. From this it follows, in particular, that all the results in recent papers \cite{BBF-13}, \cite{BBF-15}, \cite{B-09}, \cite{BB-13} using these theorems without mentioning  assumption (HL) remain valid, because in the cited papers it is enough to deal with associated radial weights which are always log-convex.
\end{rem}

\section{Boundedness of the integration operator}

In this section we study the integration operator $I: H_w(G) \to H_v(G)$:
$$
If(z): = \int_{\ell[z_0,z]} f(\zeta)d\zeta,\ \ z \in G,
$$
where $z_0$ is a fixed point of a simply connected domain $G$ and the integration is taken over a rectifiable curve $\ell[z_0,z]$ connecting $z_0$ and $z$.
In case $0\in G$ we can assume without loss of generality that $z_0=0$.

Similarly as in Section~2, we develop a new, more elementary than in \cite{HL-08}, method which guarantees that some previous results concerning the boundedness of the integration operator $I$ are true for log-convex weights. Moreover, we construct some examples showing that the log-convexity of weights is an essential assumption (see Example~\ref{ex-int} and Remark~\ref{rem-int}). In addition, we show that \cite[Proposition~2.2 (a)]{HL-08} is true without any restrictions (see Remark~\ref{rem-l}).

As in Section~2, we start with the following simple sufficient condition for the boundedness of the operator $I$ in the general case.
\begin{prop}\label{prop-di}
Suppose that $v$ and $w$ are two weights on $G$ with
$$
\sup_{z \in G} \dfrac{w_{in}(z)}{v(z)} < \infty,
$$
where
$$
w_{in}(z):=\inf_{\ell[z_0,z]}\int_{\ell[z_0,z]} w(\zeta)\,|d\zeta| \  \ (z\in G)
$$
and the infimum is taken over all rectifiable curves $\ell[z_0,z]$ connecting $z_0$ and $z$.
Then the operator $I: H_w(G) \to H_v(G)$ is continuous.
\end{prop}

\begin{proof}
For arbitrary $f \in H_w(G)$, $z \in G$ and rectifiable curve $\ell[z_0,z]$ we have
$$
|If(z)| = \left|\int_{\ell[z_0,z]} f(\zeta)d\zeta\right| \leq
\int_{\ell[z_0,z]} |f(\zeta)||d\zeta|\le
\|f\|_w \int_{\ell[z_0,z]}w(\zeta)d\zeta.
$$
From this it follows that
$$
|If(z)|\le M\|f\|_wv(z)\ \ \text{ for all } z\in G \ \text{ and } f\in H_w(G),
$$
where $M:=\displaystyle\sup_{z \in G} w_{in}(z)/v(z)$.
Thus, $\|If\|_v \leq M \|f\|_w,\ \forall f \in H_w(G)$, which completes the proof.
\end{proof}

\begin{cor}\label{cor-su-i}
Let $v$ and $w$ be two radial weights on $G=\mathbb D$ or $G=\mathbb C$ and, as before,  $a = 1$ for $G = \mathbb{D}$  and $a = +\infty$ for $G = \mathbb{C}$.
Consider the following conditions:
\begin{enumerate}
\item[(i)]
$ \displaystyle \limsup_{r \to a^-} \dfrac{w(r)}{v'(r)} < \infty.$
\item[(ii)]
$ \displaystyle \limsup_{r \to a^-} \dfrac{1}{v(r)}\int_0^rw(t)\,dt < \infty.$
\item[(iii)]
The integration operator $I: H_w(G) \to H_v(G)$ is continuous.
\end{enumerate}
Then (i)$\Longrightarrow$(ii)$\Longrightarrow$(iii).
In particular, for every radial weight $v$ on $\mathbb D$ the integration operator $I: H_v(\mathbb D)\to H_v(\mathbb D)$ is continuous.
\end{cor}

\begin{proof}
(i)$\Longrightarrow$(ii): By the L'Hospital rule.

(ii)$\Longrightarrow$(iii): By Proposition~\ref{prop-di}.

Next, for every radial weight $v$ on $\mathbb D$, $\int_0^r v(t)\,dt\le rv(r)\le v(r)$ for all $r\in[0,1)$.
Consequently, condition (ii) holds  for $w(r)=v(r)$, which implies (iii).
\end{proof}

\begin{rem}\label{rem-l}
(1) In \cite[Proposition~2.2 (a)]{HL-08} it was proved that (i)$\Longrightarrow$(iii) under the additional assumption that $H_v(G)$ is isomorphic to $l_{\infty}$.

(2) It is easy to see that, in general, (ii) does not imply (i).
For instance, consider the following radial weight $v(r):=v(r_n)+\gamma_n (r-r_n), \ r_n<r\le r_{n+1}, \ n\in\mathbb N_0: = \mathbb{N} \cup \{0\}$, where $r_0=0$, $r_n\uparrow 1$, $v(r_0)=1$ and $\gamma_{2n}(r_{2n+1}-r_{2n})\to +\infty$ as $n\to\infty$, while $\gamma_{2n-1}=1$ for all $n\in\mathbb N$.
Then, for $t_n:=(r_{2n-1}+r_{2n})/2$,
$$
\dfrac{v(t_n)}{v'(t_n)}\ge v(r_{2n-1})\to +\infty \ \text{ as } n\to \infty.
$$
Thus,  condition (i) does not hold for $w(r) = v(r)$, while (ii) holds trivially by the proof of Corollary~\ref{cor-su-i}.
In particular, this shows that  \cite[Proposition~2.2(b)]{HL-08} might fail for a radial weight of a general type.
\end{rem}

In view of Corollary~\ref{cor-su-i} and Remark~\ref{rem-l} it is natural to discuss when the condition (ii) in Corollary~\ref{cor-su-i} is necessary for the boundedness of the corresponding integration operator.

We say that a radial weight $w$ on $G=\mathbb D$ or $G=\mathbb C$ belongs to the class $R_G$ if there exist $c>0$ and $r_0\in(0,a)$ (as before, $a=1$ or $a=+\infty$) such that for every $r\in(r_0,a)$ there is a function $f_r$ in the unit ball $B_w(G)$ with
\begin{equation}\label{eq-saw}
f_r(t)\ge cw(t) \ \ \text{ for all } t\in[r_0,r].
\end{equation}

\begin{prop}\label{prop-nci}
Let $G=\mathbb D$ or $G=\mathbb C$,  $w\in R_G$ and $v$ be an arbitrary radial weight on $G$.
The integration operator $I: H_w(G)\to H_v(G)$ is continuous if and only if condition (ii) of Corollary~\ref{cor-su-i} holds.
\end{prop}

\begin{proof}
In view of Corollary~~\ref{cor-su-i} it is enough to prove the necessity. Let $I: H_w(G)\to H_v(G)$ be continuous. Then, for some $A\in(0,\infty)$,
\begin{equation}\label{eq-ivw}
\|If\|_v\le A \ \ \text{ for all } f\in B_w(G).
\end{equation}
Using that $w\in R_G$ and taking $r_0, c$ and $f_r$ as in the definition of $R_G$, we have for $r\in(r_0,a)$
\begin{eqnarray*}
A&\ge&
\dfrac{1}{v(r)}\left|\int_0^r f_r(t)\,dt\right|\ge
\dfrac{1}{v(r)}\int_{r_0}^r f_r(t)\,dt - \dfrac{r_0w(r_0)}{v(r)}\\
&\ge&
\dfrac{c}{v(r)}\int_{r_0}^r w(t)\,dt - \dfrac{r_0w(r_0)}{v(r)}\\
&\ge&
\dfrac{c}{v(r)}\int_{0}^r w(t)\,dt - \dfrac{(c+1)r_0w(r_0)}{v(r)}.
\end{eqnarray*}
This implies that
\[
\limsup_{r \to a^-} \dfrac{1}{v(r)}\int_0^rw(t)\,dt \le \dfrac{A}{c}< \infty,
\]
which completes the proof.
\end{proof}

In Clunie--K\"ov\'ari \cite{ClKo-68} for $G=\mathbb C$ and in Horowitz \cite{Hor-95} and Bierstedt--Bonet--Taskinen \cite{BBT-98} for $G=\mathbb D$ there were established some conditions on a radial weight $w$ under which there exist $f\in B_w(G)$, $r_0\in[0,a)$ and $c>0$ such that
\[
f(t)\ge cw(t) \ \ \text{ for all } t\in[r_0,a).
\]
Denote by $CR_G$ the family of all radial weights on $G$ having the last property.
Clearly, $CR_G\subset R_G$.

By \cite[Theorem~4]{ClKo-68}, every log-convex weight $w$ on $\mathbb C$ can be represented in the form
\begin{equation}\label{eq-lcw}
w(r)=w(1)\exp\int_1^r \dfrac{\omega(\rho)}{\rho}\,d\rho, \ \ r\ge1,
\end{equation}
with some positive increasing function $\omega$.
We say that $w$ is a \textit{CK--weight} on $\mathbb C$ if $w$ is log-convex and, for $\omega$ as in \eqref{eq-lcw}, there is $c>1$ with $\omega(cr)-\omega(r)\ge 1$ for all $r\ge1$.
By \cite[Theorem~4]{ClKo-68} (see also the proof of Proposition~3.1(b) in \cite{BBT-98}), for every CK--weight $w$ there exists an entire function $f(z)=\sum_{n=0}^\infty a_nz^n$ with $a_n\ge0 \ (n\in\mathbb N_0)$ such that
$$
1\le\dfrac{w(r)}{M(f,r)}\le C\ \text{ for some } C \ge 1 \text{ and all } r\ge 1.
$$
As usual, $M(f,r):=\max_{|z|=r}|f(z)|$.
Hence, $f\in B_w(\mathbb C)$ and, since $M(f,r)=\sum_{n=0}^\infty a_nr^n=f(r)$, \ $f(r)\ge w(r)/C$ for all $r\ge1$.
Consequently, $w\in CR_{\mathbb C}$ and, applying Proposition~\ref{prop-nci}, we obtain

\begin{cor}\label{cor-cic}
Let $w$ be a CK--weight and $v$ an arbitrary radial weight on $\mathbb C$.
Then the integration operator $I: H_w(\mathbb C)\to H_v(\mathbb C)$ is continuous if and only if condition (ii) of Corollary~\ref{cor-su-i} holds.
\end{cor}

We say that a radial function $w$ on $\mathbb D$ is a \textit{BBT--weight} if $w$ can be represented in the form
\begin{equation}\label{eq-bbtw}
w(r)=w(0)\exp\int_0^r \dfrac{\tau(\rho)}{1-\rho}\,d\rho, \ \ 0\le r<1,
\end{equation}
where $\tau$ is a positive increasing function on $[0,1)$ such that, for some $c>1$ and all $\rho\in[0,1)$, $\tau(1-(1-\rho)/c)-\tau(\rho)\ge1$.
It is easy to see that every function of the form \eqref{eq-bbtw} is log-convex on $[0,1)$.
Using \cite[p.~160--161]{BBT-98} and arguing as above for a CK--weight, we have that every BBT--weight belongs to $CR_{\mathbb D}$.

Next, let $w:[0,1)\to(0,\infty)$ be an increasing log-convex function.
Following \cite[Definition~4]{Hor-95}, $\log w$ is called \textit{rapidly growing} if $\log w$ is not of moderate growth (see Remark~\ref{r-mgw}) and $\log w$ is a linear function of $\log r$ between the points $r_n=\exp(-2^{-n}) \ (\approx 1-2^{-n})$.
We say that $w$ is an \textit{H--weight} on $\mathbb D$ if $\log w$ is of moderate growth or rapidly growing.

\begin{lem}\label{lem-Hor}
Each H--weight on $\mathbb D$ belongs to $CR_{\mathbb D}$.
\end{lem}
\begin{proof}
By the proofs of \cite[Theorems 2 and 3]{Hor-95}, for every H--weight $w$ on $\mathbb D$ there exists a function $g\in H(\mathbb D)$ having the form
$$
g(z)=\prod_{j=1}^\infty\dfrac{1+e^{n_j}z^{M_j}}{1+e^{-n_j}z^{M_j}}, \ \ n_j, M_j\in\mathbb N,
$$
and satisfying the following conditions:
\begin{enumerate}
\item[1)] $\log|g(z)|\le\log w(|z|)+O(1), \ \ z\in\mathbb D$.
\medskip
\item[2)] $\displaystyle\sum_{|z_n|<r}\dfrac{r}{|z_n|}=\log w(r)+O(1), \ r\in(0,1)$, where  $(z_n)_n$ is the zero sequence of $g$.
\item[3)] $M(g,r)=g(r)$ for all $r\in[0,1)$.
\end{enumerate}
By Jensen's formula and conditions 2), 3) we have that
$$
\log g(r)=\log M(g,r)\ge \sum_{|z_n|<r}\dfrac{r}{|z_n|}=\log w(r)+O(1), \ r\in(0,1).
$$
This and condition 1) imply that, for some constants $C>1, c>0$ and $r_0\in[0,1)$, the function $f(z):=g(z)/C$ belongs to $B_w(\mathbb D)$ and $f(r)\ge cw(r)$ for all $r\in[r_0,1)$. Consequently, $w\in CR_{\mathbb D}$.
\end{proof}

Thus, we have shown that every BBT-weight or H--weight on $\mathbb D$ belongs to $CR_{\mathbb D}$.
Using Proposition~\ref{prop-nci}, we then get
\begin{cor}\label{cor-cid}
Let $w$ be a BBT--weight or an H--weight on $\mathbb D$ and $v$ an arbitrary radial weight on $\mathbb D$.
Then the integration operator $I: H_w(\mathbb D)\to H_v(\mathbb D)$ is continuous if and only if condition (ii) of Corollary~\ref{cor-su-i} holds.
\end{cor}

Now we shall study the boundedness of the integration operator in the cases when $w(r) = v(r)/(1-r)$ on $\mathbb{D}$ or $w = v$ on $\mathbb{C}$.
They correspond to those studied in the previous section for the differentiation operator.
We start with the complex plane case which is simpler than the unit disc one.

\begin{thm}\label{thm-cric}
Let $v$ be a log-convex weight on $\mathbb C$. Then the following assertions are equivalent:
\begin{enumerate}
\item[(i)]
$ \displaystyle \liminf_{r \to \infty} \dfrac{v'(r)}{v(r)} >0.$
\item[(ii)]
$ \displaystyle \limsup_{r \to \infty} \dfrac{1}{v(r)}\int_0^r v(t)\,dt < \infty.$
\item[(iii)]
The integration operator $I: H_v(\mathbb C)\to H_v(\mathbb C)$ is continuous.
\end{enumerate}
\end{thm}
\begin{proof}
By Corollary~\ref{cor-su-i} (i) $\Longrightarrow$ (ii) $\Longrightarrow$ (iii).

(iii) $\Longrightarrow$ (i). Let $I: H_v(\mathbb C)\to H_v(\mathbb C)$ be continuous.
Then, for some $C>0$, $\|If\|_v\le C\|f\|_v$ for all $f\in H_v(\mathbb C)$.
Using this for $f(z)=z^n\in H_v(\mathbb C)$, we have that
\begin{equation}\label{eq-znv}
\dfrac{1}{n+1}\|z^{n+1}\|_v\le C\|z^n\|_v \ \ \text{ for all } n\in\mathbb N_0.
\end{equation}
Next, for every $n\in\mathbb N_0$,
\begin{eqnarray*}
\|z^n\|_v&=&
\exp\sup_{r>0}(n\log r-\log v(r))\\
&=&\exp\sup_{x\in\mathbb R}(nx-\varphi_v(x))=:\exp A_n,
\end{eqnarray*}
where $\varphi_v(x):=\log v(e^x)$.

Since $v$ is a log-convex radial weight on $\mathbb C$, the function $\varphi_v$ is convex on $\mathbb R$ and $x=o(\varphi_v(x))$ as $x\to+\infty$.
Without loss of generality we can assume that $\varphi_v$ is differentiable on $\mathbb R$.
Then, for every $n\in\mathbb N_0$, there exists $x_n\in\mathbb R$ with
$$
A_n= nx_n-\varphi_v(x_n) \ \ \text{ and } \ \ \varphi'_v(x_n)=n.
$$
Hence, from \eqref{eq-znv} it then follows that
$$
A_{n+1}-A_n\le\log C(n+1)=\log C(\varphi'_v(x_n)+1).
$$
Obviously,
$$
A_{n+1}-A_n\ge (n+1)x_n-\varphi_v(x_n) - A_n=x_n.
$$
Therefore, $x_n\le \log C(\varphi'_v(x_n)+1)$ or
\begin{equation}\label{eq-phiv}
\varphi'_v(x_n)e^{-x_n}\ge\frac 1C- e^{-x_n}.
\end{equation}
Since $\varphi_v$ is convex, $\varphi'_v$ is increasing.
This implies that $x_n<x_{n+1}$, $n\in \mathbb N_0$, and $x_n\to+\infty$ as $n\to\infty$.
Then, for every $x\ge x_0$, there is $n\in\mathbb N_0$ such that $x_n\le x < x_{n+1}$.
Using \eqref{eq-phiv}, we then have
\begin{align*}
\varphi'_v(x)e^{-x} &
\ge\varphi'_v(x_n)e^{-x_{n+1}}
=\frac{n}{n+1}\varphi'_v(x_{n+1})e^{-x_{n+1}} \\
& \ge\dfrac{n}{n+1} \left( \dfrac{1}{C} - e^{-x_{n+1}} \right),
\end{align*}
which implies that
$$
\liminf_{x\to+\infty}\frac{\varphi'_v(x)}{e^x}\ge\frac 1C>0.
$$
Clearly, the last condition is equivalent to (i).
\end{proof}

The following example shows that Theorem~\ref{thm-cric} might fail for radial weights of a general type.

\begin{exa}\label{ex-int}
Let $(\varepsilon_k)_k$ be a sequence of positive numbers such that $\varepsilon_1 < 1$, $\varepsilon_k \downarrow 0$ as $k \to\infty$ and $\sum_k e^k (e^{\varepsilon_k}-1) <\infty$.
For instance, one can take $\varepsilon_k = e^{-2k}$.

Put $\varphi(x): = e^x$ for $x \in (-\infty,1]$ and denote
$$
\Delta_k: = e^{k+\varepsilon_k} - (e^k + \varepsilon_k), k\in\mathbb N.
$$
Putting
$$
\varphi(x): = x -1 + \varphi(1) \ \ \text{ for } x\in(1, 1+\varepsilon_1]
$$
and
$$
\varphi(x): = e^x - \Delta_1 \ \ \text{ for } x \in (1+\varepsilon_1,2]
$$
and assuming that the function $\varphi$ is already defined on
$(-\infty,n]$, $n\in\mathbb N$, define it on $(n,n+1]$ in the following way:
$$
\varphi(x): = x -n + \varphi(n) \ \ \text{ for } x\in(n,n+\varepsilon_n]
$$
and
$$
\varphi(x): = e^x - \sum_{k=1}^n \Delta_k \ \ \text{ for } x\in(n+\varepsilon_n,n+1].
$$

It is easy to see that $\varphi$ is continuous and increasing but not convex on $\mathbb R$.
From the definition of $\varphi$ it follows that
\begin{equation}\label{eq-ephi}
e^x - C\leq \varphi(x) \leq e^x,\ \ \text{ for all } x \in \mathbb{R},
\end{equation}
where $C: =\sum_k\Delta_k = \sum_k \left(e^k (e^{\varepsilon_k}-1)-\varepsilon_k\right) <\infty$.
By making slight changes of $\varphi$ in the neighborhood of each angular point we may assume that $\varphi$ is differentiable on $\mathbb R$.

Consider two radial weights on $\mathbb C$, $v(r):=\exp\varphi(\log r)$ and $\overline{v}(r): = e^r$.
Condition \eqref{eq-ephi} implies that $e^{-C}\overline{v}(r) \leq v(r) \leq \overline{v}(r)$ for all $r >0$.
Consequently,  $H_v(\mathbb{C}) = H_{\overline{v}}(\mathbb{C})$ as sets and topologically.

Since
$$
\liminf_{r\to \infty}\dfrac{v'(r)}{v(r)} = \liminf_{x\to \infty}\dfrac{\varphi'(x)}{e^x} = \lim_{n\to\infty}\dfrac{\varphi'(n+\varepsilon_n/2)}{e^{n+\varepsilon_n/2}}= \lim_{n\to\infty}\dfrac{1}{e^{n+\varepsilon_n/2}}= 0,
$$
condition (i) of Theorem \ref{thm-cric} does not hold for $v$.
On the other hand, $\displaystyle \lim_{r\to \infty} \overline{v}'(r) / \overline{v}(r) = 1$, which implies the continuity of $I: H_{\overline{v}}(\mathbb{C}) \to H_{\overline{v}}(\mathbb{C})$ or, equivalently, $I: H_v(\mathbb{C}) \to H_v(\mathbb{C})$.
\end{exa}

\begin{cor}\label{cor-epic}
Let $v$ be a log-convex weight on $\mathbb{C}$. The following assertions are equivalent:
\begin{enumerate}
\item[(i)] The operator $D: H_v(\mathbb{C}) \to H_v(\mathbb{C})$ is an epimorphism.
\item[(ii)] $ \displaystyle 0< \liminf_{r \to \infty} \dfrac{v'(r)}{v(r)} \leq \limsup_{r \to \infty} \dfrac{v'(r)}{v(r)} <\infty.$
\item[(iii)] There exist $A, C\ge1$ such that
$$
\dfrac 1A e^{r/C}\le v(r)\le Ae^{Cr} \text{ for all } r\ge 0.
$$
\end{enumerate}
\end{cor}

\begin{proof}
(i)$\Longleftrightarrow$(ii): Follows directly from Theorems~\ref{thm-d-c} and \ref{thm-cric}.

(ii)$\Longrightarrow$(iii): Obvious.

(iii)$\Longrightarrow$(ii): In view of Theorem~\ref{thm-d-c} it is enough to check that $ \displaystyle \liminf_{r \to \infty} \dfrac{v'(r)}{v(r)}>0$.
Taking $\alpha \in (0, C^{-2}) $ and using that $v$ is log-convex and (iii), we have
$$
\dfrac{v'(r)}{v(r)}
\geq  \dfrac{\log v(r) - \log v(\alpha r)}{r\log (1 / \alpha)}\ge \dfrac{(1/C-\alpha C)r-2\log A}{r\log (1 / \alpha)}.
$$
Then
$$
\liminf_{r \to \infty} \dfrac{v'(r)}{v(r)} \geq \dfrac{1}{\log (1 / \alpha)} \left( \dfrac{1}{C} - \alpha C \right) >0.
$$
\end{proof}
\begin{rem}\label{rem-int}
It is easy to see that Example~\ref{ex-2} with $a_n = 3^{-n}$ and $b_n = \log(1+\frac{1}{n}) - 3^{-n}$,  $n \in \mathbb{N}$, shows that our assumption that $v$ is log-convex is essential in Corollary~\ref{cor-epic} (cf. Harutyunyan--Lusky \cite[Theorem~4.2]{HL-08}).
\end{rem}

Consider now the case of the unit disc.

\begin{prop}\label{prop-sci}
Let $v$ be a radial weight on $\mathbb D$ and $w(r):=v(r)/(1-r)$.
Consider the following conditions:
\begin{enumerate}
\item[(i)] $v$ satisfies
\begin{equation}\label{eq-infv}
\liminf_{r\to1^-}\dfrac{(1-r)v'(r)}{v(r)}>0.
\end{equation}
\item[(ii)] $v$ satisfies
\begin{equation}\label{eq-intv}
\limsup_{r\to1^-}\dfrac{1}{v(r)}\int_0^r\dfrac{v(t)}{1-t}\,dt<\infty.
\end{equation}
\item[(iii)] The integration operator $I: H_w(\mathbb D)\to H_v(\mathbb D)$ is continuous.
\item[(iv)]  $\log \dfrac{1}{1-r} = O (\log v(r))$ as $r \to 1^-$.
\end{enumerate}
Then (i)$\Longrightarrow$(ii)$\Longrightarrow$(iii)$\Longrightarrow$(iv).
\end{prop}
\begin{proof}
(i)$\Longrightarrow$(ii)$\Longrightarrow$(iii): Follows immediately from Corollary~\ref{cor-su-i}.

(iii)$\Longrightarrow$(iv):
Since $I: H_w(\mathbb{D}) \to H_v(\mathbb{D})$ is continuous, there exists a constant $C>0$ such that
\begin{equation}\label{eq-coni}
\|If\|_v \leq C \|f\|_w\ \ \text{ for all } f \in H_w(\mathbb{D}).
\end{equation}

Define the sequence $(f_n)_n$ in the following way:
$$
f_0(z): = \dfrac{1}{1-z},\ \ \text{ and } \ \ f_n(z): = \dfrac{1}{1-z} If_{n-1}(z),\ \ n\in\mathbb N, \ z\in\mathbb D.
$$
Obviously, $f_0 \in H_w(\mathbb{D})$.
Then $If_0 \in H_v(\mathbb{D})$ and $f_1 \in H_w(\mathbb{D})$.
By induction we have that $f_n \in H_w(\mathbb{D})$ for all $n \in \mathbb{N}$.
Moreover, using \eqref{eq-coni}, we get
\begin{align*}
\|f_n\|_w &= \sup_{z \in \mathbb{D}}\dfrac{|(1-z)^{-1}If_{n-1}(z)|}{(1-|z|)^{-1}v(|z|)} \leq \sup_{z \in \mathbb{D}}\dfrac{|If_{n-1}(z)|}{v(|z|)}\\
& = \|If_{n-1}\|_v \leq C \|f_{n-1}\|_w, \ \forall n\in\mathbb N.
\end{align*}
Therefore, $\|f_{n-1}\|_w \leq C^{n-1}\|f_0\|_w$, and, consequently, for all $n \in \mathbb{N}$,
$$
\|If_{n-1}\|_v \leq C^n \|f_0\|_w.
$$
Clearly, $If_{n-1}(z) = \dfrac{1}{n!}\left( \log \dfrac{1}{1-z} \right)^n,\ \forall n \in \mathbb{N}$.
Thus,
$$
\dfrac{1}{n!}\left( \frac{1}{C}\log \frac{1}{1-r} \right)^n \leq \|f_0\|_w v(r) \text{  for all n } \in \mathbb{N} \text{ and } r \in [0,1).
$$
It is easy to see that this inequality holds also for $n=0$.
Hence,
$$
A(r):=\sup_{n\in\mathbb N_0}\dfrac{1}{n!}\left( \frac{1}{C}\log \frac{1}{1-r} \right)^n \leq \|f_0\|_w v(r) \ \ \text{ for all } r\in[0,1).
$$
As is well known, $\displaystyle\lim_{R\to+\infty} \log\mu(R) / R =1$, where
$\mu(R):=\displaystyle\sup_{n\in\mathbb N_0} R^n/ n!$ is the maximum term of the Taylor series of $e^z$.
Then
$$
\lim_{r\to1^-}\dfrac{\log A(r)}{\frac{1}{C}\log\frac{1}{1-r}}=1
$$
and, consequently,
$$
\liminf_{r \to1^-} \dfrac{\log v(r)}{\log \frac{1}{1-r}} \geq
\liminf_{r\to1^-} \dfrac{\log A(r)-\log\|f_0\|_w}{\log \frac{1}{1-r}}=\dfrac 1C>0.
$$
\end{proof}

\begin{rem}\label{rem-iv}
Clearly, the weight $v(r)=1/(1-r)^{\alpha}$, $\alpha>0$, satisfies \eqref{eq-infv}.
Hence, the integration operator maps continuously $H_w(\mathbb D)$, $w(r)=1/(1-r)^{\alpha+1}$, into $H_v(\mathbb D)$.
This shows that condition (iv) in Proposition~\ref{prop-sci} is exact, i.e. in this condition one cannot replace the function $\log (1/(1-r))$ by a positive function $\varphi(r)$ with
$$
\limsup_{r\to1^-}\dfrac{\varphi(r)}{\log(1/(1-r))}=\infty.
$$
\end{rem}

We say that a radial weight $v$ on $\mathbb D$ is \textit{regular} if there exists
$$
\lim_{r \to 1^-} \dfrac{(1-r)v'(r)}{v(r)} =: L_v.
$$
Obviously, $L_v\in [0,+\infty]$.

\begin{cor}\label{cor-reg}
Let $v$ be a regular radial weight on $\mathbb{D}$ and $w(r) = v(r)/(1-r)$. The integration operator $I: H_w(\mathbb{D})\to H_v(\mathbb{D})$ is continuous if and only if $L_v>0$.
\end{cor}

\begin{proof}
If $L_v>0$, then the integration operator is continuous by Proposition~\ref{prop-sci}.

If $L_v=0$, then, by the L'Hospital rule,
$$
\lim_{r \to 1^-} \dfrac{\log v(r)}{\log \frac{1}{1-r}} = 0.
$$
Therefore, $v$ does not satisfy condition (iv) of Proposition~\ref{prop-sci}, which implies that the integration operator is not continuous.
\end{proof}

The following statement describes some properties of radial weights satisfying \eqref{eq-infv} and \eqref{eq-intv}.

\begin{lem}\label{lem-vid}
For a radial weight $v$ on $\mathbb{D}$ consider the following conditions:
\begin{enumerate}
\item[(i)] $v$ satisfies \eqref{eq-infv}.
\item[(ii)] $(1-r)^\alpha v(r)$ is increasing on $[r_0,1)$ for some $\alpha>0$ and $r_0\in[0,1)$.
\item[(iii)] $(1-r)^\alpha v(r)$ is almost increasing on $[0,1)$ for some $\alpha>0$, i.e. there exists a constant $C>0$ such that for any $r_1<r_2$ it follows that $ (1-r_1)^{\alpha} v(r_1) \leq C (1 - r_2)^{\alpha} v(r_2) $.
\item[(iv)] $1/v$ satisfies condition $(**)$ from \cite[p.~310]{L-95} and \cite[Definition~2.1]{L-96}, i.e. there is $k \in \mathbb{N}$ such that
$$
 \limsup_{n\to\infty} \dfrac{v(1-2^{-n})}{v(1-2^{-n-k})} < 1.
$$
\item[(v)]
There exists $\delta_0\in(0,1)$ such that
$$
\limsup_{r\to1^-}\dfrac{v(r)}{v\left(\dfrac{r+\delta_0}{1+\delta_0r}\right)}<1.
$$
\item[(vi)] There exists $\gamma >1$ such that
$$
\limsup_{r\to1^-}\dfrac{v(r^{\gamma})}{v(r)}<1.
$$
\item[(vii)] $v$ satisfies \eqref{eq-intv}.
\end{enumerate}
Then (i)$\Longleftrightarrow$(ii)$\Longrightarrow$(iii)$\Longleftrightarrow$
(iv)$\Longleftrightarrow$(v)$\Longleftrightarrow$(vi)$\Longleftrightarrow$(vii).
For a log-convex weight $v$ conditions (i)--(vii) are all equivalent.
\end{lem}
\begin{proof}
(i)$\Longleftrightarrow$(ii): Similarly to \cite[Theorem~3.3]{HL-08}.

(ii)$\Longrightarrow$(iii): Obvious.

(iii)$\Longleftrightarrow$(iv): Follows from \cite[Lemma~1 (b)]{DoLi-02}.

(iv)$\Longleftrightarrow$(v): We will use the ideas of Doma\'nski--Lindstr\"om (see the proof of \cite[Lemma~1 (a), (iii)$\Longleftrightarrow$(v)]{DoLi-02}).
Condition (iv) implies that, for some $n_0 \in \mathbb{N}$ and $q \in (0,1)$,
$$
v(1-2^{-n}) \leq q v(1-2^{-n-k}) \ \ \text{ for all } n \geq n_0.
$$
Take $\delta_0 \in (1-2^{-k-1},1)$ and, given  $r\ge 1-2^{-n_0}$, choose $n\ge n_0$ so that $1-2^{-n}\le r< 1-2^{-n-1}$.
It is easy to see that
$$
 1 - 2^{-n-1-k} < \dfrac{\delta_0 + r}{1 + \delta_0 r}.
$$
Consequently,
$$
v(r) \le v(1-2^{-n-1}) \leq q v(1-2^{-n-1-k}) \leq q v\left( \dfrac{\delta_0 + r}{1 + \delta_0 r} \right),
$$
which gives (v).

Conversely, take  $k \in \mathbb{N}$ so that $2^{-k} < (1-\delta_0)/(1+\delta_0)$.
By direct calculations we have that
$$
1 - 2^{-n-k} > \dfrac{(1 - 2^{-n}) + \delta_0 }{1+\delta_0(1 - 2^{-n})}\ \text{ for every } n\in\mathbb N.
$$
Then
$$
\dfrac{v(1-2^{-n})}{v(1-2^{-n-k})} \leq \dfrac{v(1-2^{-n})}{v \left( \dfrac{(1 - 2^{-n}) + \delta_0 }{1+\delta_0(1 - 2^{-n})} \right)}.
$$
and, applying (v), we obtain (iv).

(v)$\Longleftrightarrow$ (vi): Similarly to (v)$\Longleftrightarrow$ (vi) in the proof of Lemma~\ref{lem-d}.

(vi)$\Longrightarrow$(vii): Condition (vi) implies that there exist $q,r_0\in(0,1)$ such that
\begin{equation}\label{eq-vn}
v(r^\gamma)\le qv(r) \ \ \text{ for all } r\in[r_0,1).
\end{equation}
Without loss of generality we can assume that $\gamma\in\mathbb N$, $\gamma\ge 2$.
Given $r\in[r_0,1)$, denote by $n$ the integral part of $\log_{\gamma}\dfrac{\log r_0}{\log r}$.
Then $r_0\le r^{\gamma^n}<r_0^{1/\gamma}$.
Using \eqref{eq-vn} several times, we have that
$$
v(r^{\gamma^k})\le q^kv(r)\ \ \text{ for all } r\ge r_0 \ \text{ and } k=1,\ldots,n.
$$
Consequently,
\begin{eqnarray*}
& &\int_{r_0}^{r}\dfrac{v(t)}{1-t}\,dt \le
\int_{r_0}^{r_0^{1/\gamma}}\dfrac{v(t)}{1-t}\,dt+\sum_{k=1}^n\int_{r^{\gamma^{n-k+1}}}^{r^{\gamma^{n-k}}}\dfrac{v(t)}{1-t}\,dt
\\ &\le& \int_{r_0}^{r_0^{1/\gamma}}\dfrac{v(t)}{1-t}\,dt+\sum_{k=1}^n v(r^{\gamma^{n-k}})\log\dfrac{1-r^{\gamma^{n-k+1}}}{1-r^{\gamma^{n-k}}}\\
&\le& \int_{r_0}^{r_0^{1/\gamma}}\dfrac{v(t)}{1-t}\,dt+v(r)\sum_{k=1}^n q^{n-k}\log\left(1+r^{\gamma^{n-k}}+\ldots+r^{\gamma^{n-k}(\gamma-1)}\right)\\
&\le& \int_{r_0}^{r_0^{1/\gamma}}\dfrac{v(t)}{1-t}\,dt+v(r)\dfrac{\log \gamma}{1-q}.
\end{eqnarray*}
From this it follows that
$$
\limsup_{r\to 1^-}\dfrac{1}{v(r)}\int_0^r\dfrac{v(t)}{1-t}\,dt\le \dfrac{\log \gamma}{1-q}<\infty,
$$
that is, (vii) holds.

(vii)$\Longrightarrow$ (vi):
From (vii) it follows that there exists $A<\infty$ such that
$$
\dfrac{1}{v(r)}\int_0^r\dfrac{v(t)}{1-t}\,dt\le A \ \ \text{ for all } r\in[0,1).
$$
Then, for any $\gamma>1$, $\gamma\in\mathbb N$, and all $r\in[0,1)$,
$$
A\ge\dfrac{1}{v(r)}\int_{r^\gamma}^r\dfrac{v(t)}{1-t}\,dt
\ge\dfrac{v(r^{\gamma})}{v(r)}\log(1+r+\ldots+r^{\gamma-1}).
$$
Hence,
$$
\limsup_{r\to1^-}\dfrac{v(r^\gamma)}{v(r)}\le\dfrac{A}{\log \gamma}.
$$
Taking $\gamma>e^A$, we obtain (vi).

Thus, (iii)$\Longleftrightarrow$
(iv)$\Longleftrightarrow$(v)$\Longleftrightarrow$(vi)$\Longleftrightarrow$(vii) and, to complete the proof, it remains to check that (iv)$\Longrightarrow$ (i) for a log-convex weight $v$.

Condition (iv) implies that, for some $B>1$ and $n_0\in\mathbb N$,
$$
\dfrac{v(1-2^{-n-k})}{v(1-2^{-n})}\ge B>1 \ \ \text{ for all } n\ge n_0.
$$
Using that $v$ is log-convex, for every  $n\ge n_0+k$ and all $r\in[1-2^{-n}, 1-2^{-n-1})$, we have
\begin{align*}
r\,(\log v(r))'&\ge
(1-2^{-n})(\log v)'(1-2^{-n}) \\
& \ge
\dfrac{\log(v(1-2^{-n})/v(1-2^{-n+k}))}{\log((1-2^{-n})/(1-2^{-n+k}))} \\
&\ge \dfrac{\log B}{\log\left(1+\frac{2^{-n+k}-2^{-n}}{1-2^{-n+k}}\right)} \\
& \ge
\dfrac{(1-2^{-n+k})\log B}{2^{-n}(2^k-1)}> \dfrac{\log B}{2^{k+3} (1-r)}.
\end{align*}
Consequently,
$$
\liminf_{ r \to 1^-} \dfrac{(1-r)v'(r)}{v(r)}\ge \dfrac{\log B}{2^{k+3}}>0.
$$
\end{proof}

\begin{cor}\label{cor-vid}
Let $v$ be a log-convex radial weight on $\mathbb D$ and $w(r):=v(r)/(1-r)$.
If one of the equivalent conditions (i)--(vii) of Lemma~\ref{lem-vid} holds, then the integration operator $I: H_w(\mathbb D)\to H_v(\mathbb D)$ is continuous.
\end{cor}

\begin{rem}
In contrast to the case of entire functions we could not prove that condition \eqref{eq-infv} is always necessary for the continuity of the integration operator $I: H_w(\mathbb D)\to H_v(\mathbb D)$, where $v$ is log-convex and $w(r)=v(r)/(1-r)$.
\end{rem}

\begin{cor}
Let $v$ be a log-convex radial weight on $\mathbb D$ and $w(r)=v(r)/(1-r)$. The operator $D: H_v(\mathbb{D}) \to H_w(\mathbb{D})$ is an epimorphism if and only if
\begin{equation}\label{eq-epi}
0< \displaystyle \liminf_{r \to 1^-} \dfrac{(1-r)v'(r)}{v(r)} \leq \limsup_{r \to 1^-} \dfrac{(1-r)v'(r)}{v(r)}< \infty.
\end{equation}
\end{cor}
\begin{proof}
The sufficiency follows immediately from Theorem~\ref{thm-d-d} and Corollary~\ref{cor-vid}.

Let now the operator $D: H_v(\mathbb{D}) \to H_w(\mathbb{D})$ is an epimorphism.
Then it is continuous and, by Theorem~\ref{thm-d-d}, the right hand side of \eqref{eq-epi} holds and $v$ is of moderate growth in the sense of Horowitz.
Moreover, it is easy to see that $w$ is also log-convex. Consequently, $w$ is also of moderate growth.
Since the differentiation operator is an epimorphism, the integration operator $I:H_w(\mathbb D)\to H_v(\mathbb D)$ is continuous.
By Corollary~\ref{cor-cid}, this and the above-mentioned properties of $w$ imply that condition \eqref{eq-intv} holds.
It remains to use Lemma~\ref{lem-vid} to obtain that \eqref{eq-infv} holds, which gives the left hand side of \eqref{eq-epi}.
\end{proof}

\textbf{Acknowledgements.}
The authors are greatly indebted to Professor W.~Lusky for his explanations concerning the general assumption (HL) used in \cite{HL-08}. The research of A.V. Abanin was supported by Russian Foundation for Basic Research under Project 15-01-01404 a. The research of Pham Trong Tien was supported by NAFOSTED under grant No. 101.02-2014.49. The final version of the article was written when the second author visited Vietnam Institute for Advanced Study in Mathematics (VIASM). He would like to thank the VIASM for hospitality and support.


\end{document}